\documentclass[a4paper,12pt]{amsart}
\usepackage{amsmath,amsthm,amssymb}
\usepackage{hyperref}

\allowdisplaybreaks[1]

\newcommand{\ti}{\widetilde}
\newcommand{\la}{\lambda}

\newcommand{\ze}{\zeta}
\newcommand{\ity}{\infty}
\newcommand{\C}{\mathbb{C}}

\newcommand{\N}{\mathbb{N}}

\newcommand{\B}{\Big}

\newcommand{\G}{\mathfrak{G}}

\numberwithin{equation}{section}
\newtheorem{theorem}{Theorem}[section]
\newtheorem{lemma}[theorem]{Lemma}
\newtheorem{corollary}[theorem]{Corollary}
\theoremstyle{remark}
\newtheorem{remark}[theorem]{Remark}
\newtheorem{example}[theorem]{Example}
\newtheorem{definition}[theorem]{Definition}

\begin{document}
\title[semigroups of transcendental functions]{semigroups of transcendental entire functions and their dynamics}
\author[D. Kumar]{Dinesh Kumar}
\address{Department of Mathematics, University of Delhi,
Delhi--110 007, India}

\email{dinukumar680@gmail.com }
\author[S. Kumar]{Sanjay Kumar}

\address{Department of Mathematics, Deen Dayal Upadhyaya College, University of Delhi,
New Delhi--110 015, India }

\email{sanjpant@gmail.com}

\begin{abstract}
We investigate the dynamics of  semigroups of transcendental entire functions using Fatou-Julia theory. Several results of the dynamics associated with iteration of a transcendental entire function have been extended to transcendental  semigroups. We provide some condition for  connectivity of the Julia set of the transcendental semigroups. We also study finitely generated transcendental semigroups, abelian transcendental semigroups and limit functions of transcendental semigroups on its invariant Fatou components.
\end{abstract}



\maketitle

\section{introduction}\label{sec1}

Let $f$ be a transcendental entire function and for $n\in\N,$ let $f^n$ denote the $n$-th iterate of $f.$ The set $F(f)=\{z\in\C : \{f^n\}_{n\in\N}\,\text{ is normal in some neighborhood of}\, z\}$ is called the Fatou set of $f$ or the set of normality of $f$ and its complement $J(f)$ is called the Julia set of $f$. An introduction to the basic properties of these sets can be found in \cite{berg1}. The escaping set of $f$ denoted by $I(f)$ is the set of points in the complex plane that tend to infinity under iteration of $f$. 
The set $I(f)$ was studied for the first time by Eremenko \cite {e1} who proved that $I(f)$ is non empty, and each component of $\overline{I(f)}$ is unbounded.

A natural generalization of the dynamics associated to the iteration of a complex function is the dynamics of composite of two or more such functions and this leads to the realm of semigroups of rational and transcendental entire functions. The seminal work in this direction was done by Hinkkanen and Martin \cite{martin} related to semigroups of rational functions. In their  paper, they extended the classical theory of the dynamics  associated to the iteration of a rational function of one complex variable to a more general setting of an arbitrary semigroup of rational functions. Many of the results were extended to semigroups of transcendental entire functions in \cite{{cheng}, {dinesh1}, {poon1}, {zhigang}}. It should be noted that Sumi  has done an extensive work in the semigroup theory of rational functions and holomorphic maps. He has written a series of papers, for instance, \cite{sumi1, sumi2}. 

A transcendental semigroup $G$ is a semigroup generated by a family of transcendental entire functions $\{f_1,f_2,\ldots\}$ with the semigroup operation being functional composition. Denote the semigroup by $G=[f_1,f_2,\ldots].$ Thus each $g\in G$ is a transcendental entire function and $G$ is closed under composition.
The  Fatou set $F(G)$ of a transcendental semigroup $G$, is the largest open subset of $\C$ on which the family of functions in $G$ is normal and the Julia set $J(G)$ of $G$ is the complement of $F(G),$ that is, $J(G)=\ti\C\setminus F(G).$ The semigroup generated by a single function $g$ is denoted by $[g].$ In this case we denote $F([f])$ by $F(f)$ and $J([f])$ by $J(f)$ which are the respective Fatou set and Julia set in the classical Fatou-Julia theory of iteration of a single transcendental entire function.

 The dynamics of a semigroup is more complicated than those of a single function.  For instance, $F(G)$ and $J(G)$ need not be completely invariant and $J(G)$ may not be the entire complex plane $\C$ even if $J(G)$ has an interior point, \cite{martin}. The authors initiated the study of escaping sets of semigroups of transcendental entire functions. They generalized  the dynamics of a transcendental entire function on its escaping set to the dynamics of semigroups of transcendental entire functions on their escaping sets \cite{dinesh3}.

In this paper,  the dynamics of a transcendental entire function on its Fatou set has been generalized to the dynamics of semigroups of transcendental entire functions using Fatou-Julia theory.  We have provided some  condition for  connectivity of the Julia set of the transcendental semigroups. 
We also investigate  finitely generated transcendental semigroups and limit functions of semigroups on its invariant Fatou components. Some of the results have been illustrated by examples.

The following definitions are well known in  transcendental semigroup theory.
\begin{definition}\label{sec1,defn1}
Let $G$ be a transcendental semigroup. A set $W$ is \emph{forward invariant} under $G$ if $g(W)\subset W$ for all $g\in G$ and $W$ is \emph{backward invariant} under $G$ if $g^{-1}(W)=\{w\in\C:g(w)\in W\}\subset W$ for all $g\in G.$ $W$ is called \emph{completely invariant} under $G$ if it is both forward and backward invariant under $G.$ 
\end{definition}
It is easily seen for a transcendental semigroup $G,$ $F(G)$ is forward invariant and $J(G)$ is backward invariant, \cite[Theorem 2.1]{poon1}.

\begin{definition}\label{sec2,defna}
For a transcendental semigroup $G,$  define the \emph{backward orbit} of a point  $w$ to be the set
\[O^-(w)=\{z\in \C:\,\text{there exist}\,g\in G\,\,\text{such that}\,g(z)=w\}.\]
The \emph{Fatou exceptional} value of $G$ is defined by
\[FV(G)=\{w\in\C: O^-(w)\,\text{is finite}\}\] and contains at most one element \cite{poon1}.
\end{definition}

\begin{definition}\label{sec1,defn2}
A component $U$ of $F(G)$ is called a \emph{wandering domain} of $G$ if the set $\{U_g:g\in G\}$ is infinite (where  $U_g$ is the component of $F(G)$ containing $g(U)$). Otherwise, $U$ is called a \emph{pre-periodic} component of $F(G).$
\end{definition}

\section{multiply connected components and the  julia set}\label{sec2}
As this work is in continuation of our previous work \cite{dinesh1}, for the sake of continuity we reproduce some definitions and results of \cite{dinesh1}.
For an invariant component $U\subset F(G),$ the \emph{stabilizer} of $U$ is the set $G_U=\{g\in G:U_g=U\},$ where  $U_g$ denotes  the component of $F(G)$ containing $g(U).$
If $G$ is a transcendental semigroup and $U$  a multiply connected component of $F(G),$ define 
\[
\ti\G_U=\{g\in G: F(g)\,\text{has a multiply connected component}\, \ti U_g\supset U\}.
\]
We reproduce the  proof of non emptiness of $\ti\G_U.$ 

Suppose that $U\subset F(G)$ is a multiply connected component. Then  $U$ is bounded. Let $\gamma\subset U$ be a curve which is not contractible in $U$ and whose interior contains points of $J(G).$ We need to  show the existence of a $g\in G$ such that $J(g)$ intersects the bounded interior portion  of $\gamma$. Denote the bounded interior portion of $\gamma$ by $\gamma_1.$ Let $\zeta_0\in J(G)\cap\gamma_1.$ From \cite[Theorem 4.2]{poon1}, we have \,$J(G)=\overline{(\bigcup_{g\in G}J(g))}.$  It can be easily seen that there is a sequence $\{g_j\}$ in  $G$ such that there are points $\zeta_j\in J(g_j)$ with  $\zeta_j\to\zeta_0.$ Choose a $g_{j_0}$ from this sequence $\{g_j\}$. Since $\gamma\subset U\subset F(g_{j_{0}})$ and $\gamma_1\cap J(g_{j_0})\neq \emptyset,$ $\gamma$ is not contractible with respect to $F(g_{j_0}),$ and hence the component $\ti{U_{g_{j_0}}}$ of $F(g_{j_0})$ which contains $U$ is multiply connected. This proves that $\ti\G_U$ is non empty. 

\begin{remark}\label{sec2,rem1'}
It was shown in \cite{dinesh1} that for a transcendental semigroup $G$, a multiply connected component of $F(G)$ is bounded and wandering, and hence a pre-periodic component of $F(G)$ must be simply connected.
\end{remark}

For a transcendental semigroup $G$, the notion of  escaping set of $G,$ denoted by $I(G)$ was introduced in \cite{dinesh3}, and was   defined as
\[I(G)=\{z\in\C\,|\,\text{every sequence in}\,G\,\text{has a subsequence which diverges to infinity at}\, z\}.\]
\begin{remark}\label{sec2,1'''}
As a consequence of the definition of $I(G),$ one observes that $z\in I(G)$ if every sequence in $G$ diverges to $\ity$ at $z.$
\end{remark}
The next result provides condition for a Fatou component of $G$ to be contained in its escaping set $I(G).$ We will need the following lemma 

\begin{lemma}\cite[Theorem 4.1(i)]{dinesh1}\label{sec2,lem4}
Let $G$ be a transcendental semigroup and $U$  a multiply connected component of $F(G).$ Then for all $g\in\ti\G_U,$ $g^n\to\ity$ locally uniformly on $U.$
\end{lemma}

\begin{theorem}\label{sec2,thmc}
Let  $G=[g_1, g_2,\ldots]$  be a transcendental semigroup. 
\begin{enumerate}
\item\ If $U\subset F(G)$ is a multiply connected component, then $U\subset I(G)$;

\item\ If $U\subset F(G)$ is a Baker domain, then $U\subset I(G).$

\end{enumerate}
\end{theorem}

\begin{proof}
\begin{enumerate}
\item\ For  $g\in\ti\G_U,$ let $\ti U_g $ be a multiply connected component of $F(g)$ containing $U.$ Then from Lemma \ref{sec2,lem4},  $g^n\to\ity$ locally uniformly on $U.$ Since $ U\subset F(G)$ therefore, by normality every sequence in $G$ has a subsequence which tends to $\ity$ locally uniformly on $U$ and thus $U\subset I(G).$

\item\ From classification of periodic components of $F(G)$ \cite{dinesh1}, for all $g\in G_U, U\subset U^g\subset F(g),$ where $U^g\subset F(g)$ is a Baker domain and so $U^g\subset I(g)$ for all $g\in G_U.$ This implies that for all $g\in G_U, g^n\to\ity$ on $U$ and hence by normality, $U\subset I(G).$ \qedhere
\end{enumerate}
\end{proof}

We now provide some sufficient condition for the Julia set of a transcendental semigroup to be connected. Some of the results are motivated from Kisaka \cite{kis1}. The connectivity of $J(G)\cup\{\ity\}$ in $\ti\C$ is investigated first. Here compactness of $J(G)\cup\{\ity\}$ in $\ti\C$ simplifies the problem. To prove the result, we state a lemma 

\begin{lemma}\cite[Proposition 5.15]{beardon}\label{sec2,lemc'}
Let $K$ be a compact subset of $\ti\C.$ Then $K$ is connected if and only if each component of the complement $K^c$ is simply connected.
\end{lemma}

\begin{theorem}\label{sec2,thmb'}
Let  $G=[g_1, g_2,\ldots]$  be a transcendental semigroup. Then  $J(G)\cup\{\ity\}$ in $\ti\C$ is connected if and only if $F(G)$ has no multiply connected components.
\end{theorem}

\begin{proof}
As $J(G)\cup\{\ity\}$ is a compact subset of $\ti\C$, on applying Lemma \ref{sec2,lemc'} we get the forward implication. From Remark \ref{sec2,rem1'}, the  pre-periodic components of $F(G)$ are simply connected. Hence if a Fatou component $V$ is not simply connected, then $V$ is bounded and wandering  which implies that $J(G)\cup\{\ity\}$ is not connected in $\ti\C.$ This completes the proof.
\end{proof}

\begin{remark}\label{sec2,rem3'}
The theorem generalizes Kisaka's result \cite[Theorem 1]{kis1} to transcendental semigroups.
\end{remark}

As a consequence of Theorem \ref{sec2,thmb'} and \cite[Theorem 4.2, Theorem 4.3]{dinesh1}
 the following result is immediate

\begin{corollary}\label{sec2,corc'}
Let  $G=[g_1, g_2,\ldots]$  be a transcendental semigroup. Under each of the following conditions  $J(G)\cup\{\ity\}$ in $\ti\C$ is connected:
\begin{enumerate}
\item[(i)] $G$ is bounded on some curve $\Gamma$ tending to $\ity,$ (where  by the boundedness of $G$ on a set $A$ we mean,  all the generators in $G$ are bounded on $A$);
\item[(ii)] $F(G)$ has an unbounded component;
\item[(iii)] $F(G)$ has a domain which is completely invariant under some $g_0\in G.$
\end{enumerate}
\end{corollary}

\begin{remark}\label{sec2,rem4'}
For a transcendental semigroup $G,$ if $J(G)\cup\{\ity\}$ is disconnected, then all components of $F(G)$ are bounded, some of which are multiply connected components.
\end{remark}

We next consider the connectivity of $J(G)$ in $\C$ itself. To prove the results we will need the following lemmas

\begin{lemma}\cite[Proposition 1]{kis1}\label{sec2,lemd'}
Let $F$ be a closed subset of $\C.$ Then $F$ is connected if and only if the boundary $\partial U$ of each component $U$ of the complement $F^c$ is connected.
\end{lemma}

\begin{lemma}\cite[Proposition 5.1.4]{beardon}\label{sec2,leme'}
A domain $D$ is simply connected if and only if its boundary $\partial D$ is connected.
\end{lemma}

\begin{theorem}\label{sec2,thme'}
For a transcendental semigroup  $G=[g_1, g_2,\ldots],$ if all the  components of $F(G)$ are bounded and simply connected, then $J(G)$ is connected in $\C.$ 
\end{theorem}

\begin{proof}
In view of Lemma \ref{sec2,lemd'}, it suffices to show that the boundary $\partial U$ is connected for each component $U\subset F(G).$ Since each Fatou component  of $G$ is bounded and simply connected, from Lemma \ref{sec2,leme'}, boundary $\partial U$ is connected for each component $U\subset F(G)$ and this completes the proof.
\end{proof}

\begin{remark}\label{sec2,rem6'}
If $J(G)$ is connected and all Fatou components of $G$ are bounded, then they are also simply connected.
\end{remark}

The following result is immediate from Theorems \ref{sec2,thmb'} and \ref{sec2,thme'}

\begin{corollary}\label{sec2,cori'}
Let  $G=[g_1, g_2,\ldots]$  be a transcendental semigroup. If all components of $F(G)$ are bounded, then $J(G)$ is connected in $\C$ if and only if $J(G)\cup\{\ity\}$ is connected in $\ti\C.$
\end{corollary}

\begin{theorem}\label{sec2,thmd'}
 For a transcendental semigroup $G=[g_1, g_2,\ldots]$, if $F(G)$ has a multiply connected component $U,$ then for all  $g\in\ti\G_U, J(g)$ is disconnected in $\C.$
\end{theorem}

\begin{proof}
As $F(G)$ has a multiply connected component $U,$  for all $g\in\ti\G_U, F(g)$ will have a multiply connected component $\ti{U_g}\supset U$ and therefore,  $J(g)$ is disconnected in $\C$ for all $g\in\ti\G_U.$
\end{proof}

\begin{remark}\label{sec2,rem5'}
$J(G)$ is disconnected in $\C$ in view of Remark \ref{sec2,rem6'}. Also $I(g)$ is connected for all $g\in\ti\G_U$ by \cite[Theorem 4.6]{dinesh3}.
\end{remark}

Recall that a complex number $w\in\C$ is a critical value of a transcendental entire function $f$ if there exist some $w_0\in\C$ with $f(w_0)=w$ and $f'(w_0)=0.$ Here $w_0$ is called a critical point of $f.$ The image of a critical point of $f$ is  critical value of $f.$ Also  $\zeta\in\C$ is an asymptotic value of a transcendental entire function $f$ if there exist a curve $\Gamma$ tending to infinity such that $f(z)\to \zeta$ as $z\to\ity$ along $\Gamma.$ The set of  asymptotic values of a transcendental entire function $f$ will be denoted by $AV(f).$
The definitions of  critical point, critical value and asymptotic value of a transcendental semigroup $G$ were introduced in \cite{dinesh1}. The set of  asymptotic values of a transcendental semigroup $G$ will be denoted by $AV(G).$

For a transcendental entire function having a multiply connected component, its Fatou set does not contain any asymptotic values \cite{Hua}. This result has a partial extension to transcendental semigroups.

\begin{theorem}\label{sec2,thma'}
Let  $G=[g_1, g_2,\ldots]$  be a transcendental semigroup.
If $U\subset F(G)$ is a multiply connected component, then  $F(\ti\G_U)$ does not contain any asymptotic values of $\ti\G_U.$
\end{theorem}

\begin{proof}
 For all $g\in\ti\G_U,$ $U\subset\ti U_g\subset F(g)$ where $\ti U_g\subset F(g)$ is multiply connected. From \cite[p.\ 72]{Hua}, $F(g)$ does not contain any asymptotic values of $g$ for all $g\in\ti\G_U.$ If $z_0\in F(\ti\G_U)$ is an asymptotic value of $\ti\G_U,$ then $z_0$ is an asymptotic value of $g$ for some $g\in\ti\G_U$ a contradiction and hence the result.
\end{proof}

For a transcendental entire function having all its Fatou components bounded, the Fatou exceptional value belongs to the Julia set \cite[p.\ 129]{Hua}. This result gets generalized to the semigroups.

\begin{theorem}\label{sec2,thm4}
Let  $G=[g_1, g_2,\ldots]$  be a transcendental semigroup. If for some $g_0\in G,$ all components of $F(g_0)$ are bounded, then $FV(G)\subset J(G).$
\end{theorem}

\begin{proof}
Let $z_0\in FV(G)\cap F(G).$ Then $z_0\in FV(g)\cap F(g)$ for all $g\in G.$ In particular, $z_0\in FV(g_0)\cap F(g_0)$. This contradicts the fact that $F(g_0)$ has no unbounded component  \cite[p.\ 129]{Hua}.
\end{proof}

\section{finitely generated transcendental semigroups}\label{sec3}

We now consider finitely generated transcendental semigroups. A semigroup $G=[g_1,\ldots,g_n]$ generated by finitely many transcendental entire functions is called finitely generated. Furthermore, if $g_i$ and $g_j$ are permutable that is, $g_i\circ g_j=g_j\circ g_i$ for all $1\leq i,j\leq n$, then $G$ is called finitely generated abelian transcendental semigroup.
Recall the
 Eremenko-Lyubich class
 \[\mathcal{B}=\{f:\C\to\C\,\,\text{transcendental entire}: \text{Sing}{(f^{-1})}\,\text{is bounded}\},\]
where Sing($f^{-1}$) is the set of critical values and asymptotic values of $f$ and their finite limit points. Each $f\in\mathcal B$ is said to be of bounded type. Moreover,  if $f$ and $g$ are of bounded type, then so is $f\circ g$ \cite{berg4}. 
A transcendental entire function $f$ is of finite type if Sing($f^{-1}$) is a finite set.

The following result which gives a criterion for the connectedness of the Julia set of a transcendental entire function was proved in \cite{dom1}.
\begin{theorem}\label{sec2,thmi}\cite[Theorem 2.3]{dom1}
For any transcendental entire function $f$ of finite type, $J(f)$ is connected in $\ti\C.$
\end{theorem}

\begin{remark}\label{sec2,rema}
The result can be extended to transcendental entire functions of bounded type. The argument used in the proof of Theorem \ref{sec2,thmi} can be used verbatim to arrive at the conclusion.
\end{remark}

The next result is a generalization of the result mentioned in above remark to transcendental semigroups. We will need the following lemma to prove this result

\begin{lemma}\cite{dom1}\label{sec2,lem3}
If $A$ and $B$ are two components of a closed set $F$ in $\ti\C,$ then there is a polygon in the complement $F^c$ separating $A$ and $B$.
\end{lemma}

\begin{theorem}\label{sec2,thm9}
If $G=[g_1,\ldots, g_n]$  is a finitely generated  transcendental semigroup in which each $g_i,\, i=1,\ldots,n$ is of bounded type, then $J(G)$ is connected in $\ti\C.$
\end{theorem}

\begin{proof}
For each $g\in G, g\in\mathcal B$ and so the forward orbit of points in $F(g)$ does not approach $\ity,$ \cite[Theorem 1]{el2}. Suppose that $J(G)$ is not connected in $\ti\C.$ Let $A$ be a component of $J(G)\subset\ti\C$ which is distinct from the component $B$ of $J(G)$ which contains $\ity.$ From Lemma \ref{sec2,lem3}, there is a polygon $\gamma$ in $F(G)$ which separates $A$ and $B$. Let $U\subset F(G)$ be a multiply connected component which contains $\gamma.$ Then from Lemma \ref{sec2,lem4}, for all $g\in\ti\G_U,$ $g^n\to\ity$ locally uniformly on $U$ which is a contradiction and hence the result. 
\end{proof}

\begin{remark}\label{sec2,rem2'}
The conclusion of Theorem \ref{sec2,thm9}, in general, holds for a finitely generated abelian transcendental semigroup $G$ in which $\ity$ is not a limit function of any subsequence in $G$ in a component of $F(G).$
\end{remark}
We illustrate Theorem \ref{sec2,thm9} with an example.
\begin{example}\label{sec2,eg1}
Let $f=e^{\la z},\,\la\in\C\setminus\{0\},$  $g=f^s+\frac{2\pi i}{\la},\,s\in\N$ and  $G=[f, g].$ For $n\in\N,\,g^n=f^{ns}+\frac{2\pi i}{\la}$ and so  $J(g)=J(f).$  Observe that for $ l,m,n,p\in\N, f^l \circ g^m=f^{l+ms}$ and $g^n\circ f^p=f^{ns+p}+\frac{2\pi i}{\la}.$ Therefore for any $h\in G,$ either  $h=f^k$ for some $k\in\N,$ or $h=f^{qs}+\frac{2\pi i}{\la}=g^q$ for some $q\in\N.$ In either of the cases, $J(h)=J(f)$ and hence $J(G)=J(h)$ for all $h\in G.$ From Theorem \ref{sec2,thmi}, $J(h)$ is connected in $\ti\C$ for all $h\in G,$ and hence $J(G)$ is connected in $\ti\C.$
\end{example}

The next result gives a criterion when Fatou set of a transcendental semigroup does not contain any asymptotic values. To prove the result, we will need the following lemma
\begin{lemma}\cite[Theorem 2.1(ii)]{dinesh2}\label{sec2,lem1}
Let $f$ and $g$ be  permutable transcendental entire functions. Then $F(f\circ g)\subset F(f)\cap F(g).$
\end{lemma}

\begin{theorem}\label{sec2,thm3}
For a finitely generated abelian transcendental semigroup $G=[g_1,\ldots, g_n],$ if every Fatou component of   $g_i,\, i=1,\ldots,n$ is bounded, then $F(G)$ does not contain any asymptotic values of $G.$
\end{theorem}

\begin{proof}
 For all $g\in G,$ every Fatou component of $g$ is bounded  using  Lemma \ref{sec2,lem1}. Also  every component of $F(G)$ is bounded. Suppose that $z_0\in F(G)\cap AV(G).$ Then $z_0\in F(g)$ for all $g\in G$ and $z_0\in AV(h)$  for some $h\in G$ and hence $z_0\in F(h)\cap AV(h)$ for some $h\in G.$ But this contradicts the fact that $F(h)$ does not contain any asymptotic values of $h$ \cite[p.\ 72]{Hua}.
\end{proof}

Recall that a fixpoint $w_0$ of a meromorphic function $f$ is called \emph{weakly repelling} if $|f'(w_0)|\geq 1$ \cite{sis}. It was shown in \cite{berg2} that a transcendental entire function having finitely many weakly repelling fixed points has no multiply connected components. This result gets generalized to  transcendental  semigroups.

\begin{theorem}\label{sec2,thm5}
Let  $G=[g_1,\ldots, g_n]$  be a finitely generated  transcendental semigroup. If $G$ has only finitely many weakly repelling fixed points then $G$ has no multiply connected components.
\end{theorem}

\begin{proof}
As $G$ has only finitely many weakly repelling fixed points, each $g\in G$ can have at most finitely many weakly repelling fixed points. So for each $g\in G, F(g)$ has no multiply connected components \cite{berg2}. If $U\subset F(G)$ is a multiply connected component, then for all $g\in\ti\G_U,\, F(g)$ has a multiply connected component $\ti U_g$ which contains $U$, which is a contradiction and this completes the proof of the theorem.
\end{proof}

If a transcendental entire function is bounded on some curve tending to $\ity,$ then all its Fatou components are simply connected \cite[Corollary]{baker2}. This result gets generalized to  transcendental semigroups as seen by the following theorem

\begin{theorem}\label{sec2,thm6}
Let  $G=[g_1,\ldots, g_n]$  be a finitely generated  transcendental semigroup.  Suppose that $\Gamma_i,\,1\leq i\leq n$ are curves going to $\ity$ and $g_i\vert_{\Gamma_i}$ are bounded for $1\leq i\leq n$, then all the components of $F(G)$ are simply connected.
\end{theorem}

\begin{proof}
Observe that each $g\in G$ is bounded on one of the curves $\Gamma_i, \,1\leq i\leq n$. Thus all components of $F(g)$ are simply connected for each $g\in G,$  \cite[Corollary]{baker2}. It now follows that all components of $F(G)$ are simply connected.
\end{proof}

\begin{remark}\label{sec2,rem2}
The result, in general, holds for a transcendental semigroup $G=[g_1, g_2,\ldots]$ and a family of curves $\Gamma=\{\Gamma_i:i\in\N\}$ in which each $\Gamma_i$  tends to $\ity$ and each of the generators $g_i,\,i\in\N$ are bounded on some curve in $\Gamma.$
\end{remark}

\section{results on limit functions}

Recall that for a transcendental meromorphic function $g,$ a function $\psi(z)$ is a limit function of $\{g^n\}$ on a  component $U\subset F(g)$ if there is some subsequence of $\{g^n\}$ which converges locally uniformly on $U$ to $\psi$ \cite[p.\ 61]{Hua}. Denote by $\mathfrak{L}(U)$ all such limit functions. We now give the analogous definition of limit functions for a semigroup $G.$

\begin{definition}\label{sec2,defn1}
Let $G=[g_1,g_2,\ldots]$ be a transcendental semigroup and let $U\subset F(G)$ be a Fatou component. A function $\psi(z)$ is a limit function of $G$ on $U$ if every sequence   in $G$ has some subsequence which converges locally uniformly on $U$ to $\psi.$ Denote again by $\mathfrak{L}(U)$ all such limit functions.
\end{definition}

We next prove an important  lemma
\begin{lemma}\label{sec2,lem2} 
Let $G=[g_1,g_2,\ldots]$ be a transcendental semigroup and $U\subset F(G)$ be an invariant component. Then $\mathfrak L(U)\subset\mathfrak L(U^g),$ where  for all $g\in G_U,$ $U^g\subset F(g)$ is an invariant component containing $U.$
\end{lemma}

\begin{proof}
Suppose that $\psi\in\mathfrak L(U)$. Then every sequence  in $G$ has a subsequence which converges locally uniformly on $U$ to $\psi.$ In particular, for $g\in G_U,$ the sequence $\{g^n\}$ has a subsequence, say $\{g^{n_i}\}$ which converges locally uniformly on $U$ to $\psi.$ By normality, $\{g^{n_i}\}$ converges  locally uniformly on $U^g$ to $\psi$ and therefore, $\psi\in\mathfrak L(U^g)$.
\end{proof}

\begin{theorem}\label{sec2,thma}
Let $G=[g_1,g_2,\ldots]$ be a transcendental semigroup and $U\subset F(G)$ be an invariant component. If there exist a constant limit function $\ze,$ then either $\ze$ is a fixed point of $G$ or $\ze=\ity.$
\end{theorem}

\begin{proof}
Let $\ity\neq\ze\in\mathfrak L(U)$ be a constant limit function of $G.$ Then every sequence  in $G$ has a subsequence which converges locally uniformly on $U$ to $\ze.$ In particular, for all $g\in G,$ the sequence $\{g^n\}$ has a subsequence $\{g^{n_i}\}$ such that for $z\in U,$ $\lim g^{n_i}(z)=\ze.$ Now $g(\ze)=g(\lim g^{n_i}(z))=\lim g(g^{n_i}(z))=\ze.$ Thus $\ze$ is a fixed point of $g$ for all $g\in G$ and hence is a fixed point of $G.$
\end{proof}

The proof of next result is motivated from the following lemma
\begin{lemma}\cite[Lemma 4.3]{Hua}\label{sec2,lemx''}
Let $f$ be a transcendental entire function and $U$ be a component of $F(f).$ Then $\mathfrak L(U)$ does not contain any repelling fixed point of $f.$
\end{lemma}

\begin{theorem}\label{sec2,thmb}
Let $G=[g_1,g_2,\ldots]$ be a transcendental semigroup and $U\subset F(G)$ be an invariant component. Then $\mathfrak L(U)$ does not contain any repelling fixed point of $G_U.$
\end{theorem}

\begin{proof}
Suppose that $w\in\mathfrak L(U)$ is a repelling fixed point of $G_U.$ Then $w$ is a repelling fixed point of some $h\in G_U.$ For all $g\in G_U,\, U\subset U^g,$ where $U^g\subset F(g)$ is an invariant component containing $U.$ Since $w$ is repelling, we can choose $\la>1$ such that $|h'(w)|>\la,$ and a neighborhood $V$ of $w$ such that 
\begin{equation}\label{sec2,eq1}
|h(z)-w|=|h(z)-h(w)|\geq\la|z-w|
\end{equation}
for $z\in V.$  Using Lemma \ref{sec2,lem2}, $w\in\mathfrak L(U^g)$ for $g\in G_U.$ So for $z\in U\subset U^g,$ there exist a sequence $\{n_k\}\to\ity$ such that $\lim g^{n_k}(z)=w.$ Also there exist $N>0$ such that $g^{n_k}(z)\in V$ for all $n_k\geq N$, and  $|g(g^{n_k}(z)-w)|<|g^{n_k}(z)-w|.$ This contradicts (\ref{sec2,eq1}) and hence the result. 
\end{proof}

The next result provides a criterion for an invariant component of a transcendental semigroup to be a Siegel disk.

\begin{theorem}\label{sec2,thm7}
Let $G=[g_1,g_2,\ldots]$ be a transcendental semigroup and $U\subset F(G)$ be an invariant component. If $\mathfrak L(U)$ contains some non constant limit function, then $U$ is a Siegel disk.
\end{theorem}

\begin{proof}
For all $g\in G_U, U\subset U^g,$ where $U^g\subset F(g)$ is an invariant component containing $U.$ Let $\phi\in\mathfrak L(U)$ be a non constant limit function. Then $\phi\in\mathfrak L(U^g), g\in G_U$ using  Lemma \ref{sec2,lem2}.  Also from \cite[p.\ 64]{Hua}, $U^g$ is a Siegel disk for all $g\in G_U,$ and hence from classification of periodic components of $F(G)$ in \cite{dinesh1}, $U\subset F(G)$ is a Siegel disk.
\end{proof}

Recall that the postsingular set of an entire function $f$ is defined as 
\[\mathcal P(f)=\overline{\B(\bigcup_{n\geq 0}f^n(\text{Sing}(f^{-1}))\B)}.\]

 For a transcendental semigroup $G$, let 
\[
\mathcal P(G)=\overline{\B(\bigcup_{f\in G}{\text{Sing}}(f^{-1})\B)}.
\]
The following result rules out the existence of non constant limit functions in any component of $F(G).$ We first prove a lemma

\begin{lemma}\label{sec2,lemx}
 For a transcendental semigroup $G=[g_1,g_2,\ldots], \mathcal P(G)= \overline{\B(\bigcup_{g\in G}\mathcal P(g)\B)}.$
\end{lemma}

\begin{proof} 
We first show that $\mathcal P(G)\subset \overline{\B(\bigcup_{g\in G}\mathcal P(g)\B)}.$ For this, let $w\in \B(\bigcup_{g\in G}{\text{Sing}}(g^{-1})\B).$ Then $w\in \text{Sing}(g^{-1})$ for some $g\in G.$ This implies that $w\in\mathcal P(g)$ for some $g\in G$ and so $w\in\overline{\B(\bigcup_{g\in G}\mathcal P(g)\B)}$ which proves the first part. For the backward implication, let $z\in\B(\bigcup_{g\in G}\mathcal P(g)\B).$ Then $z\in\mathcal P(g)$ for some $g\in G.$ We now have two cases to consider.  In the first case, let $z\in\B(\bigcup_{n\geq 0}g^n(\text{Sing}(g^{-1}))\B).$ Then $z\in g^n(\text{Sing}(g^{-1}))$ for some $n\geq 0.$ From \cite{baker1}, we have  $g^n(\text{Sing}(g^{-1}))\subset \text{Sing}(g^{n+1})^{-1}$ which further is contained in $\mathcal P(G).$ Thus we get    
$\B(\bigcup_{n\geq 0}g^n(\text{Sing}(g^{-1}))\B)\subset\mathcal P(G).$  In the second case, let $z\in \B(\bigcup_{n\geq 0}g^n(\text{Sing}(g^{-1}))\B)^{\prime}.$ Then there exist a sequence $\{z_n\}\subset \B(\bigcup_{n\geq 0}g^n(\text{Sing}(g^{-1}))\B)$ such that $\lim z_n=z.$ It can be  easily seen that $\{z_n\}\subset \mathcal P(G)$ which implies that $z\in\mathcal P(G).$ So we have $\B(\bigcup_{n\geq 0}g^n(\text{Sing}(g^{-1}))\B)^{\prime}\subset\mathcal P(G).$ Combining the two cases we get the backward implication and this completes the proof of the lemma.
\end{proof}

Example \ref{sec2,eg1} with restricted values of parameter $\la$  can also be used to illustrate Lemma \ref{sec2,lemx}

\begin{example}\label{sec2,eg2}
Let $f=e^{\la z},\,0<\la<\frac{1}{e},$ $g=f+\frac{2\pi i}{\la},$ and  $G=[f,g].$
 Clearly, 
Sing$(f^{-1})=\{0\}$ and Sing$(g^{-1})=\B\{\frac{2\pi i}{\la}\B\}.$ As in Example \ref{sec2,eg1}, for any $h\in G,$ either $h=f^k$ for some $k\in\N,$ or $h=f^q+\frac{2\pi i}{\la}=g^q$ for some $q\in\N.$ It can be easily seen that $\mathcal P(g)=\mathcal P(f)+\frac{2\pi i}{\la}.$ Also from  \cite{baker1}, $\mathcal P(f^l)=\mathcal P(f)$ for all $l\in\N.$ Consequently, we have
\begin{equation}
\begin{split}\notag
\overline{\B(\bigcup_{h\in G}\mathcal P(h)\B)}
&=\overline{\B(\bigcup_{k=1}^{\ity}\B(\mathcal P(f^k)\bigcup\mathcal P\B(f^k+\frac{2\pi i}{\la}\B)\B)\B)}\\
&=\overline{\B(\bigcup_{k=1}^{\ity}\mathcal P(f^k)\B)}\bigcup\overline{\B(\bigcup_{k=1}^{\ity}\mathcal P\B(f^k+\frac{2\pi i}{\la}\B)\B)}\\
&=\mathcal P(f)\bigcup\B(\mathcal P(f)+\frac{2\pi i}{\la}\B).
\end{split}
\end{equation}
Also
\begin{equation}
\begin{split}\notag
\mathcal P(G)
&=\overline{\B(\bigcup_{h\in G}{\text{Sing}}(h^{-1})\B)}\\
&=\overline{\bigcup_{k=1}^{\ity}\B(\text{Sing}(f^k)^{-1}\bigcup\text{Sing}\B(f^k+\frac{2\pi i}{\la}\B)^{-1}\B)}\\
&=\overline{\bigcup_{k=1}^{\ity}\B(\bigcup_{s=0}^{k-1}f^s(\text{Sing}(f^{-1})\B)}\bigcup\overline{\bigcup_{k=1}^{\ity}\B(\bigcup_{s=0}^{k-1}f^s(\text{Sing}(f^{-1}))+\frac{2\pi i}{\la}\B)}\\
&=\mathcal P(f)\bigcup\B(\mathcal P(f)+\frac{2\pi i}{\la}\B).
\end{split}
\end{equation}
Hence $\mathcal P(G)=\overline{\B(\bigcup_{h\in G}\mathcal P(h)\B)}.$
\end{example}

\begin{theorem}\label{sec2,thm8}
Let  $G=[g_1,\ldots, g_n]$  be a finitely generated  transcendental semigroup. Suppose that $\mathcal P(G)$ has an empty interior and a connected complement and for all $g\in G,$ $\mathcal P(g)$ has a connected complement. Then any sequence  in $G$ has no subsequence having a non constant limit function in any component of $F(G).$
\end{theorem}

\begin{proof}
$\mathcal P(G)=\overline{\B(\bigcup_{g\in G}\mathcal P(g)\B)}$ from Lemma \ref{sec2,lemx} and since $\mathcal P(G)$ has an empty interior so does $\mathcal P(g)$ for all $g\in G.$  From \cite{baker1}, no subsequence of $\{g^n\}$ can have a non constant limit function in any component of $F(g)$ for all $g\in G.$ Suppose that a sequence $\{h_n\}$ in $G$ has a subsequence which has a non constant limit function $\phi$ in some invariant component  $U\subset F(G).$ As $\mathfrak L(U)\subset\mathfrak L(U^g),$ where 
for all $g\in G_U, U^g\subset F(g)$ is an invariant component containing $U,$ we have $\phi\in\mathfrak L(U^g)$ which is a contradiction and hence the result.
\end{proof}

The next result is a generalization of Baker's result \cite{baker1} on the location of constant limit functions.

\begin{theorem}\label{sec2,thm10}
For a   transcendental semigroup  $G=[g_1,g_2,\ldots]$, any constant limit function of $G$ on an invariant component of $F(G)$ lies in   $\mathcal P(G)\cup\{\ity\}.$
\end{theorem}

\begin{proof}
Suppose that $U\subset F(G)$ be an invariant component and  $\xi\in\mathfrak L(U)$ be a constant limit function of $G.$ From Lemma \ref{sec2,lem2}, $\mathfrak L(U)\subset\mathfrak L(U^g),$ where  for all $g\in G_U,$ $U^g\subset F(g)$ is an invariant component containing $U.$ Also from \cite{baker1}, $\xi\in\mathcal P(g)\cup\{\ity\},$ for all $g\in G_U$ and as $\mathcal P(g)\subset\mathcal P(G)$ for all $g\in G,$ we obtain $\xi\in\mathcal P(G)\cup\{\ity\}.$
\end{proof}

Acknowledgement. The author's appreciate  the referee's valuable comments.

\end{document}